\documentclass{article}

\usepackage{cite}
\usepackage{amssymb}
\usepackage{amsmath}
\usepackage{amsfonts}
\usepackage{amsthm}
\usepackage{fancyhdr}
\usepackage[explicit]{titlesec}
\usepackage{titletoc}
\usepackage{booktabs}
\usepackage[all]{xy}
\usepackage[colorlinks=true,linkcolor=blue]{hyperref}
\usepackage{enumitem} 
\usepackage{cancel}

\usepackage{url}

\usepackage{hyperref}

\newtheorem*{theorem*}{Theorem}
\newtheorem{theorem}{Theorem}[section]
\newtheorem{corollary}{Corollary}[section]
\newtheorem{definition}{Definition}[section]

\newtheorem{lemma}{Lemma}[section]

\numberwithin{equation}{section}

\newcommand{\actson}{\curvearrowright}

\everymath{\displaystyle}

\begin{document}

\title{Uniform mixing and completely positive sofic entropy}

\author{Tim Austin and Peter Burton}

\date{\today}

\maketitle

\begin{abstract} Let $G$ be a countable discrete sofic group. We define a concept of uniform mixing for measure-preserving $G$-actions and show that it implies completely positive sofic entropy. When $G$ contains an element of infinite order, we use this to produce an uncountable family of pairwise nonisomorphic $G$-actions with completely positive sofic entropy. None of our examples is a factor of a Bernoulli shift. \end{abstract}

\section{Introduction}

Let $G$ be a countable discrete sofic group, $(X,\mu)$ a standard probability space and $T: G \curvearrowright X$ a measurable $G$-action preserving $\mu$. In \cite{Bow10}, Lewis Bowen defined the sofic entropy of $(X,\mu,T)$ relative to a sofic approximation under the hypothesis that the action admits a finite generating partition. The definition was extended to general $(X,\mu,T)$ by Kerr and Li in \cite{KeLi11} and Kerr gave a more elementary approach in \cite{Ke13}. In \cite{Bow12a} Bowen showed that when $G$ is amenable, sofic entropy relative to any sofic approximation agrees with the standard Kolmogorov-Sinai entropy. Despite some notable successes such as the proof in \cite{Bow10} that Bernoulli shifts with distinct base-entropies are nonisomorphic, many aspects of the theory of sofic entropy are still relatively undeveloped.\\
\\
Rather than work with abstract measure-preserving $G$-actions, we will use the formalism of $G$-processes. If $G$ is a countable group and $A$ is a standard Borel space, we will endow $A^G$ with the right-shift action given by $(g \cdot a)(h) = a(hg)$ for $g,h \in G$ and $a \in A^G$. A $G$-process over $A$ is a Borel probability measure $\mu$ on $A^G$ which is invariant under this action. Any measure-preserving action of $G$ on a standard probability space is measure-theoretically isomorphic to a $G$-process over some standard Borel space $A$. We will assume the state space $A$ is finite, which corresponds to the case of measure-preserving actions which admit a finite generating partition. Note that by results of Seward from \cite{Sew14a} and \cite{Sew14b}, the last condition is equivalent to an action admitting a countable generating partition with finite Shannon entropy. \\
\\
In \cite{Aus15}, the first author introduced a modified invariant called model-measure sofic entropy which is a lower bound for Bowen's sofic entropy. Let $\Sigma = (\sigma_n:G \to \mathrm{Sym}(V_n))$ be a sofic approximation to $G$. Model-measure sofic entropy is constructed in terms of sequences $(\mu_n)_{n=1}^\infty$ where $\mu_n$ is a probability measure on $A^{V_n}$. If these measures replicate the process $(A^G,\mu)$ in an appropriate sense then we say that $(\mu_n)_{n=1}^\infty$ locally and empirically converges to $\mu$. We refer the reader to \cite{Aus15} for the precise definitions. We have substituted the phrase `local and empirical convergence' for the phrase `quenched convergence' which appeared in \cite{Aus15}. This has been done to avoid confusion with an alternative use of the word `quenched' in the physics literature. A process is said to have completely positive model-measure sofic entropy if every nontrivial factor has positive model-measure sofic entropy. The goal of this paper is the to prove the following theorem, which generalizes the main theorem of \cite{DGRS08}.

\begin{theorem} \label{thm1.0} Let $G$ be a countable sofic group containing an element of infinite order. Then there exists an uncountable family of pairwise nonisomorphic $G$-processes each of which has completely positive model-measure sofic entropy (and hence completely positive sofic entropy) with respect to any sofic approximation to $G$. None of these processes is a factor of a Bernoulli shift. \end{theorem}

In order to prove Theorem \ref{thm1.0} we introduce a concept of uniform mixing for sequences of model-measures. This uniform model-mixing will be defined formally in Section \ref{secmetric}. It implies completely positive model-measure sofic entropy.

\begin{theorem} \label{thm1} Let $G$ be a countable sofic group and let $(A^G,\mu)$ be a $G$-process with finite state space $A$. Suppose that for some sofic approximation $\Sigma$ to $G$, there is a uniformly model-mixing sequence $(\mu_n)_{n=1}^\infty$ which locally and empirically converges to $\mu$ over $\Sigma$. Then $(A^G,\mu)$ has completely positive lower model-measure sofic entropy with respect to $\Sigma$. \end{theorem}

As in \cite{DGRS08}, the examples we exhibit to establish Theorem \ref{thm1.0} are produced via a coinduction method for lifting $H$-processes to $G$-processes when $H \leq G$. If $(A^H,\nu)$ is an $H$-process then we can construct a corresponding $G$-process $(A^G,\mu)$ as follows. Let $T$ be a transversal for the right cosets of $H$ in $G$. Identify $G$ as a set with $H \times T$ and thereby identify $A^G$ with $(A^H)^T$. Set $\mu = \nu^T$. We call $(A^G,\mu)$ the coinduced process and denote it by $\mathrm{CInd}^G_H(\nu)$. (See page $72$ of \cite{K} for more details on this construction.) When $H \cong \mathbb{Z}$ this procedure preserves uniform mixing.

\begin{theorem} \label{thm2} Let $G$ be a countable sofic group and let $(A^\mathbb{Z},\nu)$ be a uniformly mixing $\mathbb{Z}$-process with finite state space $A$. Let $H \leq G$ be a subgroup isomorphic to $\mathbb{Z}$ and identify $A^\mathbb{Z}$ with $A^H.$ Then for any sofic approximation $\Sigma$ to $G$, there is a uniformly model-mixing sequence of measures which locally and empirically converges to $\mathrm{CInd}^G_H(\nu)$ over $\Sigma$. \end{theorem}

We remark that it is easy to see that if $(A^G,\mu)$ is a Bernoulli shift (that is to say, $\mu$ is a product measure), then there is a uniformly model-mixing sequence which locally and empirically converges to $\mu$. Indeed, if $\mu = \eta^G$ for a measure $\eta$ on $A$ then the measures $\eta^{V_n}$ on $A^{V_n}$ are uniformly model-mixing and locally and empirically converge to $\mu$. Thus Theorem \ref{thm1} shows that Bernoulli shifts with finite state space have completely positive sofic entropy, giving another proof of this case of the main theorem from \cite{Ke15}. We believe that completely positive sofic entropy for general Bernoulli shifts can be deduced along the same lines, requiring only a few additional estimates, but do not pursue the details here.

\subsection{Acknowledgements}

The first author's research was partially supported by the Simons Collaboration on Algorithms and Geometry. The second author's research was partially supported by NSF grants DMS-0968710 and DMS-1464475.

\section{Preliminaries}

\subsection{Notation}

The notation we use closely follows that in \cite{Aus15}; we refer the reader to that reference for further discussion. Let $A$ be a finite set. For any pair of sets $W \subseteq S$ we let $\pi_W: A^S \to A^W$ be projection onto the $W$-coordinates (thus our notation leaves the larger set $S$ implicit). Let $G$ be a countable group and let $(A^G,\mu)$ be a $G$-process. For $F \subseteq G$ we will write $\mu_F = \pi_{F*} \mu \in \mathrm{Prob}(A^F)$ for the $F$-marginal of $\mu$.\\
\\
Let $B$ be another finite set and let $\phi:A^G \to B$ be a measurable function. If $F \subseteq G$ we will say that $\phi$ is $F$-local if it factors through $\pi_F$. We will say $\phi$ is local if it is $F$-local for some finite $F$. Let $\phi^G: A^G \to B^G$ be given by $\phi^G(a)(g) = \phi(g \cdot a)$ and note that $\phi^G$ is equivariant between the right-shift on $A^G$ and the right-shift on $B^G$.\\
\\
Let $V$ be a finite set and let $\sigma$ be a map from $G$ to $\mathrm{Sym}(V)$. For $g \in G$ and $v \in V$ we write $\sigma^g \cdot v$ instead of $\sigma(g)(v)$. For $F \subseteq G$ and $S \subseteq V$ we define \[ \sigma^F(S) = \{\sigma^g \cdot s: g \in F, s \in S\}. \]  For $v \in V$ we write $\sigma^F(v)$ for $\sigma^F(\{v\})$. We write $\Pi_{v,F}^\sigma$ for the map from $A^V$ to $A^F$ given by $\Pi^\sigma_{v,F}(\overline{a})(g) = \overline{a}(\sigma^g \cdot v)$ for $\overline{a} \in A^V$ and $g \in F$. We write $\Pi_v^\sigma$ for $\Pi_{v,G}^\sigma$. With $\phi: A^G \to B$ as before, we write $\phi^\sigma$ for the map from $A^V$ to $B^V$ given by $\phi^\sigma(\overline{a})(v) = \phi \bigl(\Pi_v^\sigma(\overline{a}) \bigr)$.\\
\\
If $D$ is a finite set and $\eta$ is a probability measure on $D$ then $\mathrm{H}(\eta)$ denotes the Shannon entropy of $\eta$, and for $\epsilon > 0$ we define \[ \mathrm{cov}_\epsilon(\eta) = \min \bigl \{|F|: F \subseteq D \mbox{ is such that } \eta(F) > 1- \epsilon \bigr \}. \]
If $\phi:D\to E$ is a map to another finite set then we may write $\mathrm{H}_\mu(\phi)$ in place of $\mathrm{H}(\phi_\ast\mu)$.  For $p \in [0,1]$ we let $\mathrm{H}(p)= -p\log p - (1-p)\log(1-p)$.\\
\\
We use the $o(\cdot)$ and $\lesssim$ asymptotic notations with respect to the limit $n \to \infty$. Given two functions $f$ and $g$ on $\mathbb{N}$, the notation $f \lesssim g$ means that there is a positive constant $C$ such that $f(n) \leq Cg(n)$ for all $n$.

\subsection{An information theoretic estimate}

\begin{lemma} \label{lem1} Let $A$ be a finite set and let $(V_n)_{n=1}^{\infty}$ be a sequence of finite sets such that $|V_n|$ increases to infinity. Let $\mu_n$ be a probability measure on $A^{V_n}$. We have \[ \liminf_{n \to \infty} \frac{\mathrm{H}(\mu_n)}{|V_n|}  \leq \sup_{\epsilon >0} \hspace{2 pt} \liminf_{n \to \infty} \frac{1}{|V_n|} \log \mathrm{cov}_\epsilon (\mu_n). \] \end{lemma}

\begin{proof} Let $\mu$ be a probability measure on a finite set $F$ and let $E \subseteq F$. By conditioning on the partition $\{E,F\setminus E\}$ and then recalling that entropy is maximized by uniform distributions we obtain \begin{align} \mathrm{H}(\mu) & = \mu(E)\cdot \mathrm{H}(\mu(\cdot\,|\,E)) + \mu(F\setminus E)\cdot \mathrm{H}(\mu(\cdot\,|\,F\setminus E)) + \mathrm{H}(\mu(E))\nonumber\\ & \leq \mu(E) \cdot \log(|E|) + (1 - \mu(E)) \cdot \log(|F \setminus E|) + \mathrm{H}(\mu(E)). \label{eq2} \end{align}
 
Now let $\mu_n$ and $V_n$ be as in the statement of the lemma. Let $\epsilon > 0$ and let $S_n \subseteq A^{V_n}$ be a sequence of sets with $\mu_n(S_n) > 1 - \epsilon$ and $|S_n| = \mathrm{cov}_\epsilon(\mu_n)$. By applying (\ref{eq2}) with $F = A^{V_n}$ and $E = S_n$ we have \begin{align} \liminf_{n \to \infty} \frac{\mathrm{H}(\mu_n)}{|V_n|} & \leq \liminf_{n \to \infty} \frac{1}{|V_n|} \bigl( \mu(S_n) \cdot \log (|S_n|) + (1 - \mu (S_n)) \cdot \log ( \vert A^{V_n} \setminus S_n \vert) + \mathrm{H}(\mu(S_n)) \bigr) \nonumber \\ & \leq \liminf_{n \to \infty} \frac{1}{|V_n|} \bigl( \log ( |S_n|)  + \epsilon \cdot \log(\vert A^{V_n} \vert)  + \mathrm{H}(\epsilon) \bigr) \nonumber \\ &\leq \left( \liminf_{n \to \infty} \frac{1}{|V_n|} \log \mathrm{cov}_\epsilon \bigl(\mu_n)  \right) + \epsilon \cdot \log(|A|). \nonumber  \end{align} 

Now let $\epsilon$ tend to zero to obtain the lemma. \end{proof}

\section{Metrics on sofic approximations and uniform model-mixing} \label{secmetric}

Let us fix a proper right-invariant metric $\rho$ on $G$: for instance, if $G$ is finitely generated then $\rho$ can be a word metric, and more generally we may let $w:G\to [0,\infty)$ be any proper weight function and define $\rho$ to be the resulting weighted word metric. Again let $V$ be a finite set and let $\sigma$ be a map from $G$ to $\mathrm{Sym}(V)$. Let $H_\sigma$ be the graph on $V$ with an edge from $v$ to $w$ if and only if $\sigma^g \cdot v = w$ or $\sigma^g \cdot w = v$ for some $g \in G$. Define a weight function $W$ on the edges of $H_\sigma$ by setting \[ W(v,w) = \min \bigl \{ \rho(g,1_G): \sigma^g \cdot v = w \mbox{ or } \sigma^g \cdot w = v \bigr \}. \] If $v$ and $w$ are in the same connected component of $H_\sigma$ let $\rho_\sigma$ be the $W$-weighted graph distance between $v$ and $w$, that is \[   \rho_\sigma(v,w) = \min \left \{ \sum_{i=0}^{k-1} W(p_i,p_{i+1}): (v=p_0,p_1,\ldots,p_{k-1},p_k= w) \mbox{ is an } H_\sigma \mbox{-path from }v \mbox{ to }w \right \}. \] Having defined $\rho_\sigma$ on the connected components of $H_\sigma$, choose some number $M$ much larger than the $\rho_\sigma$-distance between any two points in the same connected component. Set $\rho_\sigma(v,w) = M$ for any pair $v,w$ of vertices in distinct connected components of $H_\sigma$. Note that if $(\sigma_n:G \to \mathrm{Sym}(V_n))$ is a sofic approximation to $G$ then for any fixed $r < \infty$ once $n$ is large enough the map $g \mapsto \sigma_n^g \cdot v$ restricts to an isometry from $B_\rho(1_G,r)$ to $B_{\rho_{\sigma_n}}(v,r)$ for most $v \in V_n$.\\
\\
In the sequel the sofic approximation will be fixed, and we will abbreviate $\rho_{\sigma_n}$ to $\rho_n$. We can now state the main definition of this paper.

\begin{definition} \label{def1} Let $(V_n)_{n=1}^\infty$ be a sequence of finite sets with $|V_n| \to \infty$ and for each $n$ let $\sigma_n$ be a map from $G$ to $\mathrm{Sym}(V_n)$. Let $A$ be a finite set. For each $n \in \mathbb{N}$ let $\mu_n$ be a probability measure on $A^{V_n}$. We say the sequence $(\mu_n)_{n=1}^\infty$ is \textbf{uniformly model-mixing} if the following holds. For every finite $F \subseteq G$ and every $\epsilon > 0$ there is some $r < \infty$ and a sequence of subsets $W_n \subseteq V_n$ such that \[ |W_n| = (1 - o(1)) |V_n| \] and if $S \subseteq W_n$ is $r$-separated according the metric $\rho_n$ then \[ \mathrm{H} \left( \pi_{\sigma^F_n(S) *} \mu_n \right) \geq |S| \cdot ( \mathrm{H}(\mu_F) - \epsilon ). \] \end{definition}

This definition is motivated by Weiss' notion of uniform mixing from the special case when $G$ is amenable: see \cite{Wei03} and also Section 4 of \cite{DGRS08}. Let us quickly recall that notion in the setting of a $G$-process $(A^G,\mu)$. First, if $K \subseteq G$ is finite and $S \subseteq G$ is another subset, then $S$ is \textbf{$K$-spread} if any distinct elements $s_1,s_2 \in S$ satisfy $s_1s_2^{-1} \not\in K$.  The process $(A^G,\mu)$ is \textbf{uniformly mixing} if, for any finite-valued measurable function $\phi:A^G\to B$ and any $\epsilon > 0$, there exists a finite subset $K \subseteq G$ with the following property: if $S\subseteq G$ is another finite subset which is $K$-spread, then
\[\mathrm{H}\bigl((\phi^G_\ast\mu)_S \bigr) \geq |S|\cdot (\mathrm{H}_\mu(\phi) - \epsilon).\]
Beware that we have reversed the order of multiplying $s_1$ and $s_2^{-1}$ in the definition of `$K$-spread' compared with \cite{DGRS08}.  This is because we work in terms of observables such as $\phi$ rather than finite partitions of $A^G$, and shifting an observable by the action of $g \in G$ corresponds to shifting the partition that it generates by $g^{-1}$.\\
\\
The principal result of \cite{RudWei00} is that completely positive entropy implies uniform mixing.  The reverse implication also holds: see \cite{GolSin00} or Theorem 4.2 in \cite{DGRS08}.  Thus, uniform mixing is an equivalent characterization of completely positive entropy.\\
\\
The definition of uniform mixing may be rephrased in terms of our proper metric $\rho$ on $G$ as follows.  The process $(A^G,\mu)$ is uniformly mixing if and only if, for any finite-valued measurable function $\phi:A^G\to B$ and any $\epsilon > 0$, there exists an $r < \infty$ with the following property: if $S\subseteq G$ is $r$-separated according to $\rho$, then
\[\mathrm{H}\bigl((\phi^G_\ast\mu)_S \bigr) \geq |S|\cdot (\mathrm{H}_\mu(\phi) - \epsilon).\]
This is equivalent to the previous definition because a subset $S \subseteq G$ is $r$-separated according to $\rho$ if and only if it is $B_\rho(1_G,r)$-spread. The balls $B_\rho(1_G,r)$ are finite, because $\rho$ is proper, and any other finite subset $K\subseteq G$ is contained in $B_\rho(1_G,r)$ for all sufficiently large $r$.\\
\\
This is the point of view on uniform mixing which motivates Definition \ref{def1}.  We use the right-invariant metric $\rho$ rather than the general definition of `$K$-spread' sets because it is more convenient later.\\
\\
Definition \ref{def1} is directly compatible with uniform mixing in the following sense.  If $G$ is amenable and $(F_n)_{n=1}^\infty$ is a F\o lner sequence for $G$, then the sets $F_n$ may be regarded as a sofic approximation to $G$: an element $g \in G$ acts on $F_n$ by translation wherever this stays inside $F_n$ and arbitrarily at points which are too close to the boundary of $F_n$.  If $(A^G,\mu)$ is an ergodic $G$-process, then it follows easily that the sequence of marginals $\mu_{F_n}$ locally and empirically converge to $\mu$ over this F\o lner-set sofic approximation.  If $(A^G,\mu)$ is uniformly mixing, then this sequence of marginals is clearly uniformly model-mixing in the sense of Definition~\ref{def1}.\\
\\
On the other hand, suppose that $(A^G,\mu)$ admits a sofic approximation and a locally and empirically convergent sequence of measures over that sofic approximation which is uniformly model-mixing. Then our Theorem \ref{thm1} shows that $(A^G,\mu)$ has completely positive sofic entropy.  If $G$ is amenable then sofic entropy always agrees with Kolmogorov-Sinai entropy \cite{Bow12a}, and this implies that $(A^G,\mu)$ has completely positive entropy and hence is uniformly mixing, by the result of \cite{RudWei00}.\\
\\
Thus if $G$ is amenable then completely positive entropy and uniform mixing are both equivalent to the existence of a sofic approximation and a locally and empirically convergent sequence of measures over it which is uniformly model-mixing.  If these conditions hold, then we expect that one can actually find a locally and empirically convergent and uniformly model-mixing sequence of measures over \emph{any} sofic approximation to $G$. This should follow using a similar kind of decomposition of the sofic approximants into F\o lner sets as in Bowen's proof in \cite{Bow12a}.  However, we have not explored this argument in detail.\\
\\
Definition~\ref{def1} applies to a shift-system with a finite state space.  It can be transferred to an abstract measure-preserving $G$-action on $(X,\mu)$ by fixing a choice of finite measurable partition of $X$.  However, in order to study actions which do not admit a finite generating partition, it might be worth looking for an extension of Definition~\ref{def1} to $G$-processes with arbitrary compact metric state spaces, similarly to the setting in~\cite{Aus15}.  We also do not pursue this generalization here.

\section{Proof of Theorem \ref{thm1}}

We will use basic facts about the Shannon entropy of observables (i.e. random variables with finite range), for which we refer the reader to Chapter $2$ of \cite{CovTho06}. Let $\Sigma = ( \sigma_n:G \to \mathrm{Sym}(V_n) )$, $(A^G,\mu)$ and  $(\mu_n)_{n=1}^\infty$ be as in the statement of Theorem \ref{thm1}. The following is the `finitary' model-measure analog of Lemma $5.1$ in \cite{DGRS08}. 

\begin{lemma} \label{lem4} Let $F \subseteq G$ be finite. Let $B$ be a finite set and let $\phi: A^G \to B$ be an $F$-local observable. Let $S_n \subseteq V_n$ be a sequence of sets such that $|S_n| \gtrsim |V_n|$. Then we have \[ \mathrm{H}(\mu_F) - \frac{1}{|S_n|}\mathrm{H}(\pi_{\sigma_n^F(S_n)*} \mu_n) \geq \mathrm{H}_\mu(\phi) - \frac{1}{|S_n|} \mathrm{H} \bigl(\pi_{S_n *} \phi^{\sigma_n}_* \mu_n \bigr) - o(1). \] \end{lemma}

\begin{proof}[Proof of Lemma \ref{lem4}] Let $\theta:A^F \to B$ be a map with $\theta \circ \pi_F = \phi$. Fix $n \in \mathbb{N}$ and $S \subseteq V_n$. Let $\alpha = \pi_{\sigma^F_n(S)}:A^{V_n} \to A^{\sigma^F_n(S)}$ and let $\beta = \pi_S \circ \phi^{\sigma_n}:A^{V_n} \to B^S$. For $s \in S$ let $\alpha_s = \Pi_{s,F}^{\sigma_n}:A^{V_n} \to A^F$ and let $\beta_s = \theta \circ \Pi_{s,F}^{\sigma_n}:A^{V_n} \to B$. Then we have $\alpha = (\alpha_s)_{s \in S}$ and $\beta = (\beta_s)_{s \in S}$. Enumerate $S = (s_k)_{k=1}^m$ and write $\alpha_{s_k} = \alpha_k$. All entropies in the following display are computed with respect to $\mu_n$. We have

\begin{align*} \mathrm{H}(\alpha) & = \mathrm{H}( \alpha_1,\ldots,\alpha_m) \\ &= \mathrm{H}(\alpha_1) + \sum_{k=1}^{m-1} \mathrm{H}( \alpha_{k+1} \vert \alpha_1,\ldots,\alpha_k ) \\ & = \mathrm{H}(\alpha_1,\beta_1) + \sum_{k=1}^{m-1} \mathrm{H}( \alpha_{k+1}, \beta_{k+1} \vert \alpha_1,\ldots,\alpha_k)  \\ & = \mathrm{H}(\beta_1) + \mathrm{H}(\alpha_1|\beta_1) + \sum_{k=1}^{m-1} \mathrm{H}(\beta_{k+1} \vert \alpha_1,\ldots,\alpha_k) + \sum_{k=1}^{m-1} \mathrm{H}( \alpha_{k+1} \vert \beta_{k+1},\alpha_1,\ldots,\alpha_k) \\ & \leq \mathrm{H}(\beta_1) + \sum_{k=1}^{m-1} \mathrm{H}(\beta_{k+1} \vert \beta_1,\ldots,\beta_k) + \sum_{k=1}^m \mathrm{H} ( \alpha_k | \beta_k) \\ & = \mathrm{H}(\beta) + \sum_{k=1}^m \mathrm{H}(\alpha_k | \beta_k). \end{align*}

Let $\iota$ be the identity map on $A^F$. Then \begin{align} |S| \cdot \mathrm{H}(\mu_F) - \mathrm{H}(\pi_{\sigma_n^F(S)*} \mu_n) & = |S| \cdot \mathrm{H}_{\mu_F}(\iota) - \mathrm{H}_{\mu_n}(\alpha) \nonumber \\ & \geq |S| \cdot \mathrm{H}_{\mu_F}(\theta) + |S| \cdot \mathrm{H}_{\mu_F}(\iota|\theta) - \mathrm{H}_{\mu_n} (\beta) - \sum_{s \in S} \mathrm{H}_{\mu_n}(\alpha_s | \beta_s) \nonumber \\ & = |S| \cdot \mathrm{H}_\mu(\phi) - \mathrm{H} \bigl(\pi_{S*} \phi_*^{\sigma_n} \mu_n \bigr) + |S| \cdot \mathrm{H}_{\mu_F}(\iota|\theta) - \sum_{s \in S} \mathrm{H}_{\mu_n}(\alpha_s|\beta_s). \label{eq1.1} \end{align}

Now allowing $n$ to vary, let $S_n \subseteq V_n$ be a sequence of sets such that $|S_n| \gtrsim |V_n|$. Write $\nu_n = \pi_{\sigma_n^F(S_n)*} \mu_n$. Let $s \in S_n$ be such that the obvious map from $F$ to $\sigma^F_n(s)$ is injective. Then the function $\overline{a} \mapsto \Pi^{\sigma_n}_{s,F}(\overline{a})$ provides an identification of $A^{\sigma^F_n(s)}$ with $A^F$. This identification sends $\alpha_s$ to $\iota$ and $\beta_s$ to $\theta$. When $n$ is large the $\sigma^F_n(s)$ marginal of $\mu_n$ will resemble $\mu_F$ for most $s \in S_n$. Since $\alpha_s$ and $\beta_s$ are $\pi_{\sigma^F_n(s)}$ measurable this implies that $\mathrm{H}_{\mu_F}(\iota|\theta) \approx \mathrm{H}_{\nu_n}(\alpha_s|\beta_s)$ for most $s$. More precisely, we can find a sequence of sets $C_n \subseteq S_n$ with  \[|C_n| = (1-o(1))|S_n| \] such that \[ \max_{s \in C_n} \hspace{2 pt} \bigl \vert \mathrm{H}_{\mu_F}(\iota|\theta) - \mathrm{H}_{\nu_n}(\alpha_s|\beta_s) \bigr \vert = o(1).\] Thus  \begin{align*} \left \vert |S_n| \cdot \mathrm{H}_{\mu_F}(\iota|\theta) - \sum_{s \in S_n} \mathrm{H}_{\nu_n}(\alpha_s|\beta_s) \right \vert & \leq \sum_{s \in C_n} \bigl \vert \mathrm{H}_{\mu_F}(\iota|\theta) - \mathrm{H}_{\nu_n}(\alpha_s|\beta_s) \bigr \vert + \sum_{s \in S_n \setminus C_n} \bigl \vert \mathrm{H}_{\mu_F}(\iota|\theta) - \mathrm{H}_{\nu_n}(\alpha_s|\beta_s) \bigr \vert \\ & =  o(|S_n|). \end{align*}

The lemma then follows from (\ref{eq1.1}) and the above. \end{proof}

Recall that for a measure space $(X,\mu)$ and two observables $\alpha$ and $\beta$ on $X$ the Rokhlin distance between $\alpha$ and $\beta$ is defined by \[ d^{\mathrm{Rok}}_\mu(\alpha,\beta) = \mathrm{H}_\mu(\alpha|\beta) + \mathrm{H}_\mu(\beta|\alpha). \] This is a pseudometric on the space of observables on $X$. An easy computation shows that if $\alpha_1,\ldots,\alpha_n$ and $\beta_1,\ldots,\beta_n$ are two families of observables on $X$ then \[ d^{\mathrm{Rok}}_\mu((\alpha_1,\ldots,\alpha_n),(\beta_1,\ldots,\beta_n)) \leq \sum_{k=1}^n d^{\mathrm{Rok}}_\mu(\alpha_k,\beta_k) .\]

\begin{lemma} Let $\phi,\psi:A^G \to B$ be two local observables. Let $S_n \subseteq V_n$ be a sequence of sets with $|S_n| \gtrsim |V_n|$. Then we have  \[  \frac{1}{|S_n|} \bigl \vert \mathrm{H}(\pi_{S_n *} \phi_*^{\sigma_n} \mu_n) - \mathrm{H}(\pi_{S_n *} \psi_*^{\sigma_n} \mu_n) \bigr \vert  \leq  d^{\mathrm{Rok}}_\mu (\phi,\psi) + o(1).  \] \end{lemma}

\begin{proof} Let $\alpha_n = \pi_{S_n} \circ \phi^{\sigma_n}: A^{V_n} \to B^{S_n}$ and let $\beta_n = \pi_{S_n} \circ \psi^{\sigma_n}: A^{V_n} \to B^{S_n}$. Let $F$ be a finite subset of $G$ such that both $\phi$ and $\psi$ are $F$-local. Let $\theta:A^F \to B$ be a map such that $\theta \circ \pi_F = \phi$ and let $\kappa:A^F \to B$ be a map such that $\kappa \circ \pi_F = \psi$. For $s \in S_n$ let $\alpha_{n,s} = \theta \circ \Pi_{s,F}^{\sigma_n}: A^{V_n} \to B$ so that $\alpha_n = (\alpha_{n,s})_{s \in S_n}$. Also let $\beta_{n,s} = \kappa \circ \Pi_{s,F}^{\sigma_n}: A^{V_n} \to B$. Then we have \begin{align} \frac{1}{|S_n|} \bigl \vert  \mathrm{H}(\pi_{S_n *} \phi^{\sigma_n}_* \mu_n) - \mathrm{H}(\pi_{S_n *} \psi^{\sigma_n}_* \mu_n) \bigr \vert \nonumber &= \frac{1}{|S_n|} \bigl \vert \mathrm{H}_{\mu_n}(\alpha_n) - \mathrm{H}_{\mu_n}(\beta_n) \bigr \vert \nonumber \\ & \leq \frac{1}{|S_n|} \cdot d^{\mathrm{Rok}}_{\mu_n}(\alpha_n,\beta_n) \nonumber \\ & = \frac{1}{|S_n|} \cdot d^{\mathrm{Rok}}_{\mu_n}\bigl( (\alpha_{n,s})_{s \in S_n},(\beta_{n,s})_{s \in S_n}\bigr) \nonumber \\ & \leq \frac{1}{|S_n|} \sum_{s \in S_n} d^{\mathrm{Rok}}_{\mu_n} (\alpha_{n,s},\beta_{n,s}) \label{eq2.0} \end{align}

If the map $g \mapsto \sigma_n^g \cdot s$ is injective on $F$, we can identify $A^{\sigma_n^F(s)}$ with $A^F$ and thereby identify $\alpha_{n,s}$ with $\theta$ and $\beta_{n,s}$ with $\kappa$. Note that \[ d^{\mathrm{Rok}}_{\mu_F} (\theta,\kappa)  = d^{\mathrm{Rok}}_\mu (\phi,\psi ).\] It follows that for any $\epsilon > 0$ we can find a weak star neighborhood $\mathcal{O}$ of $\mu$ such that if $s \in S_n$ is such that $(\Pi_s^{\sigma_n})_* \mu_n \in \mathcal{O}$ then \[ \Bigl \vert d^\mathrm{Rok}_{\mu_n} (\alpha_{n,s},\beta_{n,s}) - d^{\mathrm{Rok}}_\mu (\phi,\psi)  \Bigr \vert < \epsilon. \]

Thus, since $\mu_n$ locally and empirically converges to $\mu$, there are sets $C_n \subseteq S_n$ with $|C_n| = (1-o(1))|S_n|$ such that \begin{equation} \label{eq7.7} \max_{s \in C_n} \hspace{2 pt} \Bigl \vert d^{\mathrm{Rok}}_{\mu_n} ( \alpha_{n,s},\beta_{n,s}) - d^{\mathrm{Rok}}_\mu (\phi,\psi)  \Bigr \vert = o(1). \end{equation}

The lemma now follows from (\ref{eq2.0}) and (\ref{eq7.7}). \end{proof}

\begin{corollary} \label{lem6} Let $\bigl(\phi_m: A^G \to B \bigr)_{m=1}^\infty$ be a sequence of local observables and let $\phi:A^G \to B$ be a local observable. Let $S_n \subseteq V_n$ be a sequence of sets with $|S_n| \gtrsim |V_n|$. Then if $(m_n)_{n=1}^\infty$ increases to infinity at a slow enough rate we have  \[  \frac{1}{|S_n|} \bigl \vert \mathrm{H}(\pi_{S_n *} \phi_*^{\sigma_n} \mu_n) - \mathrm{H}(\pi_{S_n *} \phi_{m_n *}^{\sigma_n} \mu_n) \bigr \vert  \leq  d^{\mathrm{Rok}}_\mu (\phi,\phi_{m_n}) + o(1).  \] \end{corollary}

\begin{proof}[Proof of Theorem \ref{thm1}] Let $B$ be a finite set and let $\psi: A^G \to B$ be an observable with $\mathrm{H}_\mu(\psi) > 0$. Let $(\phi_m)_{m=1}^\infty$ be an AL approximating sequence for $\psi$ rel $\mu$ (see Definition $4.4$ in \cite{Aus15}). Then the sequence $\phi_m$ converges to $\psi$ in $d_\mu^{\mathrm{Rok}}$. In particular, $\phi_m$ is a Cauchy sequence and so we can find $M \in \mathbb{N}$ so that for all $m \geq M$ we have \begin{equation} d_\mu^{\mathrm{Rok}} (\phi_m,\phi_M) \leq \frac{\mathrm{H}_\mu(\psi)}{8}. \label{eq3.3} \end{equation}

We will also assume $M$ is large enough that \begin{equation} \label{eq3.7} \mathrm{H}_\mu(\phi_{M}) \geq \frac{\mathrm{H}_\mu(\psi)}{2}.\end{equation}  Let $F$ be a finite subset of $G$ such that $\phi_M$ is $F$-local. Then Definition \ref{def1} provides an $r < \infty$ and a sequence of subsets $W_n \subseteq V_n$ such that $|W_n| = (1-o(1))|V_n|$ and if $S \subseteq W_n$ is $r$-separated then \begin{equation} \mathrm{H}(\mu_F)- \frac{1}{|S|} \mathrm{H}(\pi_{\sigma_n^F(S) *} \mu_n) \leq \frac{\mathrm{H}_\mu(\phi_{M})}{2}. \label{eq3.5} \end{equation} 

Let $K = |B_\rho(\mathrm{1}_G,r)|$. Since $\sigma_n$ is a sofic approximation there are sets $W'_n \subseteq V_n$ with $|W'_n| = (1-o(1))|V_n|$ such that if $w \in W'_n$ then the $\rho_n$ ball of radius $r$ around $w$ has cardinality at most $K$. Write $Y_n = W_n \cap W'_n$ and note that we have $|Y_n| = (1-o(1))|V_n|$. For each $n$ let $S_n$ be an $r$-separated subset of $Y_n$ with maximal cardinality. Then $Y_n \subseteq \bigcup_{s \in S_n} B_{\rho_n}(s,r)$ so that \begin{equation} |S_n| \geq \frac{|Y_n|}{K} = (1-o(1))\frac{|V_n|}{K}. \label{eq3.9} \end{equation}

By Lemma \ref{lem4} and (\ref{eq3.5}) we have \[ \mathrm{H}_\mu(\phi_M) - \frac{1}{|S_n|} \mathrm{H} \bigl( \pi_{S_n *} \phi_{M*}^{\sigma_n} \mu_n \bigr) - o(1) \leq \frac{\mathrm{H}_\mu(\phi_M)}{2}  \]

so that from (\ref{eq3.7}) we have \begin{equation}  \label{eq3.8} \frac{\mathrm{H}_\mu(\psi)}{4} - o(1) \leq \frac{1}{|S_n|} \mathrm{H}\bigl(\pi_{S_n *} \phi_{M*}^{\sigma_n} \mu_n \bigr).  \end{equation}

By Proposition $5.15$ in \cite{Aus15} if $(m_n)_{n=1}^\infty$ increases to infinity at a slow enough rate then $(\phi_{m_n}^{\sigma_n})_* \mu_n$ will locally and empirically converge to $\psi^G_* \mu$. Since $A$ is finite, by the same argument as for Proposition $8.1$ in \cite{Aus15} we have \begin{align} \underline{\mathrm{h}}_\Sigma^{\mathrm{q}} \bigl(\psi^G_* \mu \bigr) &\geq \sup_{\epsilon >0} \liminf_{n \to \infty} \frac{1}{|V_n|} \log \mathrm{cov}_\epsilon \bigl((\phi_{m_n}^{\sigma_n})_* \mu_n \bigr) \nonumber \\ & \geq \liminf_{n \to \infty} \frac{1}{|V_n|} \hspace{2 pt} \mathrm{H}\bigl((\phi_{m_n}^{\sigma_n})_* \mu_n \bigr) \label{eq3.2} \end{align}

where the second inequality follows from Lemma \ref{lem1}. We also assume that $(m_n)_{n=1}^\infty$ increases slowly enough for Corollary \ref{lem6} to hold. By (\ref{eq3.3}) we have \[ \left \vert \frac{1}{|S_n|} \mathrm{H} \bigl(\pi_{S_n *} \phi_{M*}^{\sigma_n} \mu_n \bigr) - \frac{1}{|S_n|} \mathrm{H}\bigl(\pi_{S_n *} (\phi_{m_n}^{\sigma_n})_* \mu_n \bigr) \right \vert  \leq \frac{\mathrm{H}_\mu(\psi)}{8} + o(1).  \]

Combining this with (\ref{eq3.8}) we see that 

\[ \frac{1}{|S_n|} \mathrm{H}\bigl(\pi_{S_n *} (\phi^{\sigma_n}_{m_n})_* \mu_n \bigr) \geq \frac{\mathrm{H}_\mu(\psi)}{8} - o(1). \]

By the above and (\ref{eq3.9}) we have that for all sufficiently large $n$, \begin{equation} \label{eq3.10} \mathrm{H}\bigl( (\phi^{\sigma_n}_{m_n})_* \mu_n \bigr) \geq \frac{\mathrm{H}_\mu(\psi)}{8K+1} |V_n| \end{equation}

Theorem \ref{thm1} now follows from (\ref{eq3.2}) and (\ref{eq3.10}). \end{proof}

\section{Proof of Theorem \ref{thm2}}

Let $(A^\mathbb{Z},\nu)$ be a uniformly mixing $\mathbb{Z}$-process, and for each positive integer $l$ let $\nu_l$ be the marginal of $\nu$ on $A^l$. Let $\Sigma = (\sigma_n:G \to \mathrm{Sym}(V_n))$ be an arbitrary sofic approximation to $G$. Let $h \in G$ have infinite order and write $H = \langle h \rangle \cong \mathbb{Z}$. We construct a measure $\mu_n$ on $A^{V_n}$ for each $n \in \mathbb{N}$. We will later show that the sequence $(\mu_n)_{n=1}^\infty$ is uniformly model-mixing and locally and empirically converges to $\mu$ over $\Sigma$.\\
\\
We first construct a measure $\mu_n^l$ on $A^{V_n}$ for each pair $(n,l)$ with $l$ much smaller than $n$. For a given $n$, the single permutation $\sigma_n^h$ partitions $V_n$ into a disjoint union of cycles. Since $h$ has infinite order and $\Sigma$ is a sofic approximation, once $n$ is large most points will be in very long cycles. In particular we assume that most points are in cycles with length much larger than $l$. Partition the cycles into disjoint paths so that as many of the paths have length $l$ as possible, and let $\mathcal{P}^l_n = (P^l_{n,1},\ldots,P^l_{n,k_n})$ be the collection of all length-$l$ paths that result (so $\mathcal{P}^l_n$ is not a partition of the whole of $V_n$, but covers most of it). Fix any element $\overline{a}_0 \in A^{V_n}$ and define a random element $\overline{a} \in A^{V_n}$ by choosing each restriction $\overline{a} \upharpoonright_{P^l_{n,i}}$ independently with the distribution of $\nu_l$ and extending to the rest of $V_n$ according to $\overline{a}_0$. Let $\mu_n^l$ be the law of this $\overline{a}$.\\
\\
Now let $(l_n)_{n=1}^\infty$ increase to infinity at a slow enough rate that the following two conditions are satisfied:
\begin{enumerate}
 \item[(a)] The number of points of $V_n$ that lie in some member of the family $\mathcal{P}^{l_n}_n$ is $(1-o(1))|V_n|$.
\item[(b)] Whenever $g,g' \in G$ lie in distinct right cosets of $H$, so that $g^{-1} h^p g' \neq 1_G$ for all $p \in \mathbb{Z}$, we have
\[|\{v \in V_n:\ (\sigma_n^g)^{-1}(\sigma_n^h)^p\sigma_n^{g'}\cdot v = v\ \hbox{for some}\ p \in \{-l_n,\dots,l_n\}\}|= o(|V_n|)\]
\end{enumerate}
Set $\mu_n = \mu_n^{l_n}$. We separate the proof that $(\mu_n)_{n=1}^\infty$ has the required properties into two lemmas.

\begin{lemma} \label{lem2.1} $(\mu_n)_{n=1}^\infty$ locally and empirically converges to $\mu$ over $\Sigma$. \end{lemma}

\begin{proof}[Proof of Lemma \ref{lem2.1}] Since $(A^G,\mu)$ is ergodic, by Corollary $5.6$ in \cite{Aus15} it suffices to show that $\mu_n$ locally weak star converges to $\mu$. For a set $I \subseteq \mathbb{Z}$ write $h^I = \{h^i: i \in I\}$. Fix a finite set $F \subseteq G$. By enlarging $F$ if necessary we can assume there is an interval $I \subseteq \mathbb{Z}$ such that $F = \bigcup_{k=1}^m h^I t_k$ for $t_1,\ldots,t_m$ in some transversal for the right cosets of $H$ in $G$. For each $g \in F$ let $j_g$ be a fixed element of $A$. Let $B \subseteq A^G$ be defined by \[ B = \bigl \{ a \in A^G: a(g) = j_g \mbox{ for all } g \in F \bigr \} \] and let $\epsilon > 0$. Then sets such as \[ \mathcal{O} = \bigl \{ \eta \in \mathrm{Prob}(A^G): \eta(B) \approx_\epsilon \mu(B) \bigr \} \] form a subbasis of neighborhoods around $\mu$.  It therefore suffices to show that when $n$ is large we have $(\Pi_v^{\sigma_n})_* \mu_n \in \mathcal{O}$ with high probability in the choice of $v \in V_n$.\\
\\
For $k \in \{1,\ldots,m\}$ let \[ B_k = \bigl \{ x \in A^\mathbb{Z}: x(i) = j_{h^i t_k} \mbox{ for all } i \in I \bigr\}. \] Note that $\mu$ is defined in such a way that $\mu(B) = \prod_{i=1}^k \nu(B_k)$. Now, let $W_n$ be the set of all points $v \in V_n$ such that the following conditions hold.

\begin{enumerate}[label=(\roman*)]
\item The map $g \mapsto \sigma^g_n \cdot v$ is injective on $F$.
\item $\sigma^{h^i t_k}_n  \cdot v = (\sigma^h_n)^i \sigma^{t_k}_n \cdot v$ for all $i \in I$ and $k \in \{1,\ldots,m\}$.
\item For all pairs $g,g' \in F$, $\sigma_n^g \cdot v$ is in the same path as $\sigma_n^{g'} \cdot v$ if and only if $g$ and $g'$ lie in the same right coset of $H$. In particular, each of the images $\sigma_n^g\cdot v$ for $g \in F$ is contained in some member of $\mathcal{P}_n^{l_n}$. \end{enumerate}

We claim that $|W_n| = (1-o(1))|V_n|$. Clearly Conditions $(\mathrm{i})$ and $(\mathrm{ii})$ are satisfied with high probability in $v$, and so is the last part of Condition $(\mathrm{iii})$, by Condition (a) in the choice of $(l_n)_{n=1}^\infty$.\\
\\
Fix $g,g' \in F$ and suppose that $g$ and $g'$ are in the same coset of $H$, so that we have $g = h^i t_k$ and $g' = h^{i'} t_k$ for some $k \in \{1,\ldots,m\}$ and $i,i' \in I$. If $v$ satisfies Condition $(\mathrm{ii})$ then we have \[ (\sigma^h_n)^{i'-i} \sigma^g_n \cdot v = (\sigma^h_n)^{i'-i} (\sigma^h_n)^i \sigma_n^{t_k} \cdot v = (\sigma^h_n)^{i'} \sigma_n^{t_k} \cdot v = \sigma^{g'}_n \cdot v \] so that $\sigma^g_n \cdot v$ and $\sigma^{g'}_n \cdot v$ will lie in the same path assuming that $\sigma_n^{t_k} \cdot v$ is not one of the first or last $|I|$ elements of its path. Note that for any $v \in V_n$ we have \[ \bigl \vert \bigl \{w: \sigma^{t_k}_n \cdot w = v \mbox{ for some } k \in \{1,\dots,m\} \bigr\} \bigr \vert \leq m. \] It follows that the number of points $v \in V_n$ such that $\sigma^{t_k}_n \cdot v$ is one of the first or last $|I|$ elements of a path is at most $2m p_n |I| + o(|V_n|)$ where $p_n$ is the total number of paths in $V_n$. By Condition (a) in the choice of $(l_n)_{n=1}^\infty$, most of $V_n$ is covered by paths whose lengths increase to infinity.  Since also $p_n = o(V_n)$, it follows that $\sigma^g_n \cdot v$ lies in the same path as $\sigma^{g'}_n \cdot v$ with high probability in $v$.\\
\\
On the other hand, suppose that $g$ and $g'$ are in distinct cosets of $H$.  Assume that $\sigma^g_n \cdot v$ and $\sigma^{g'}_n \cdot v$ are in the same path. Then there is $p \in \{-l_n,\ldots,l_n\}$ with $\sigma^g_n \cdot v = (\sigma_n^h)^p \sigma^{g'}_n \cdot v$, and hence $(\sigma_n^g)^{-1}(\sigma_n^h)^p \sigma^{g'}_n \cdot v = v$. By Condition (b) in the choice of $(l_n)_{n=1}^\infty$ there are only $o(|V_n|)$ vertices $v$ for which this holds. Thus we have established the claim.\\
\\
Now let $v \in W_n$. We have \[ (\Pi_v^{\sigma_n})_* \mu_n(B) = \mu_n\bigl(\bigl\{ \overline{a} \in A^{V_n}: \overline{a}(\sigma_n^g \cdot v) = j_g \mbox{ for all } g \in F \bigr \} \bigr). \] For each $k \in \{1,\ldots,m\}$ the set $\{ (\sigma^h_n)^i \sigma^{t_k}_n \cdot v: i \in I   \}$  is contained in a single path. Since the marginal of $\mu_n$ on each path is $\nu_{l_n}$ the probability that \[ \overline{a}\bigl((\sigma^h_n)^i \sigma^{t_k}_n \cdot v \bigr) = j_{h^i t_k} \] for all $i \in I$ is equal to $\nu_{l_n}(B_k) = \nu(B_k)$. On the other hand, the marginals of $\mu_n$ on distinct paths are independent, so the probability that $\overline{a}(\sigma^g_n \cdot v) = j_g$ for all $g \in F$ is actually equal to $\prod_{i=1}^k \nu(B_k)$. \end{proof}

If $(A^\mathbb{Z},\nu)$ is weakly mixing, then so is the co-induced $G$-action. In particular, this certainly holds if $(A^{\mathbb{Z}},\nu)$ is uniformly mixing.  Therefore we may immediately promote Lemma~\ref{lem2.1} to the fact that $(\mu_n)_{n=1}^\infty$ locally and doubly empirically converges to $\mu$ over $\Sigma$, by Lemma 5.15 of~\cite{Aus15}.  In fact, we suspect that local and double empirical convergence holds here whenever $(A^{\mathbb{Z}},\nu)$ is ergodic.

\begin{lemma} \label{lem2.2} $(\mu_n)_{n=1}^\infty$ is uniformly model-mixing. \end{lemma}

\begin{proof}[Proof of Lemma \ref{lem2.2}] Let $F \subseteq G$ be finite and let $\epsilon > 0$. Again decompose $F = \bigcup_{k=1}^m h^I t_k$ for some interval $I \subseteq \mathbb{Z}$ and elements $t_k \in T$. Note that the restriction of the metric $\rho$ to $H$ is a proper right invariant metric on $H \cong \mathbb{Z}$, even though it might be different from the usual metric on $\mathbb{Z}$. Thus since $\nu$ is uniformly mixing we can find some $r_0 < \infty$ such that if $(I_j)_{j=1}^q$ is a family of intervals in $\mathbb{Z}$ which are each of length $|I|$ and are pairwise at distance at least $r_0$ then writing $K = \bigcup_{j=1}^q I_j$ we have \begin{equation} \label{eq4.0} \mathrm{H}(\nu_K) \geq q \cdot \left(\mathrm{H}(\nu_I) - \frac{\epsilon}{m} \right).  \end{equation} Let $r < \infty$ be large enough that for all $g,g' \in G$ if $\rho(g,g') \geq r$ then $\rho(fg,f'g') \geq r_0$ for all $f,f' \in F$. Such a choice of $r$ is possible since by right-invariance of $\rho$ we have $\rho(fg,g) = \rho(f,1_G)$ and $\rho(f'g',g') = \rho(f',1_G)$. Let $W_n$ be as in the proof of Lemma \ref{lem2.1} and recall that $|W_n| = (1-o(1))|V_n|$. Let $S \subseteq W_n$ be $r$-separated according to $\rho_n$.\\
\\
Fix a path $P \in \mathcal{P}^{l_n}_n$ and let $S_P$ be the set of points $v \in S$ such that $\sigma^{t_{k(v)}}_n \cdot v \in P$ for some $k(v) \in \{1,\ldots,m\}$. Since $S \subseteq W_n$, Condition $(\mathrm{iii})$ from the previous proof implies that \[ \sigma^F_n(S) \cap P = \bigcup_{v \in S_P} \{ (\sigma^h_n)^i \sigma^{t_{k(v)}}_n \cdot v: i \in I\}. \] Each of the sets in the latter union is an interval of length $|I|$ in $P$ and by our choice of $r$ these are pairwise at distance $r_0$ in $\rho_n$ restricted to $P$. Since the marginal of $\mu_n$ on $P$ is equal to $\nu_{n_l}$, (\ref{eq4.0}) implies that \[ \mathrm{H}(\pi_{(\sigma^F_n(S) \cap P)*} \mu_n) \geq |S_P| \cdot \left( \mathrm{H}(\nu_I) - \frac{\epsilon}{m} \right). \] Since the marginals of $\mu_n$ on distinct paths are independent, this implies that \begin{equation} \label{eq4.3} \mathrm{H}(\pi_{\sigma^F_n(S)*} \mu_n) \geq \left( \sum_{P \in \mathcal{P}^{l_n}_n} |S_P| \right) \cdot \left(\mathrm{H}(\nu_I) - \frac{\epsilon}{m} \right).  \end{equation} By Condition $(\mathrm{iii})$ in the definition of $W_n$, each $v \in S$ appears in $S_P$ for exactly $m$ paths $P$. Therefore \begin{equation} \label{eq4.2} \sum_{P \in \mathcal{P}^{l_n}_n} |S_P| = m \cdot |S|.  \end{equation} Now $\mathrm{H}(\mu_F) = m \cdot \mathrm{H}(\nu_I)$ so from (\ref{eq4.3}) and (\ref{eq4.2}) we have \[ \mathrm{H}(\pi_{\sigma^F_n(S)*} \mu_n) \geq |S| \cdot (\mathrm{H}(\mu_F) - \epsilon) \] as required. \end{proof}

\begin{proof}[Proof of Theorem \ref{thm2}] Theorem \ref{thm2} now follows from Theorem \ref{thm1} and Lemmas \ref{lem2.1} and \ref{lem2.2}. \end{proof}

\section{Proof of Theorem \ref{thm1.0}}

\begin{proof}[Proof of Theorem \ref{thm1.0}] This part of the argument is essentially the same as the corresponding part of \cite{DGRS08}. Consider the family of uniformly mixing $\mathbb{Z}$-processes $\{(4^\mathbb{Z},\nu_\omega): \omega \in 2^\mathbb{N} \}$ constructed in Section $6$ of \cite{DGRS08}. Fix an isomorphic copy $H$ of $\mathbb{Z}$ in $G$ and let $\mu_\omega = \mathrm{CInd}_H^G(\nu_\omega)$. By Theorems \ref{thm1} and \ref{thm2} the process $(4^G,\mu_\omega)$ has completely positive model-measure sofic entropy. Note that the restriction of the $G$-action to $H$ is a permuted power of the original $\mathbb{Z}$-process in the sense of Definition $6.5$ from \cite{DGRS08}. Thus by Proposition $6.6$ in that reference, the processes $ \{ (4^G,\mu_\omega):\omega \in 2^\mathbb{N}  \}$ are pairwise nonisomorphic.\\
\\
Suppose toward a contradiction that for some $\omega$, $(4^G,\mu_\omega)$ is a factor of a Bernoulli shift $(Z^G,\eta^G)$ over some standard probability space $(Z,\eta)$. Let $\psi: Z^G \to 4^G$ be an equivariant measurable map with $\psi_* \eta^G = \mu_\omega$. Note that the restricted right-shift action $H \actson (Z^G,\eta^G)$ is still isomorphic to a Bernoulli shift and $\psi$ is still a factor map from this process to the restricted action $H \actson (4^G,\mu_\omega)$. Thus the latter $\mathbb{Z}$-process is isomorphic to a Bernoulli shift and so is its factor $(4^\mathbb{Z},\nu_\omega)$. This contradicts Corollary $6.4$ in \cite{DGRS08}. \end{proof}

\bibliographystyle{plain}
\bibliography{bibliography}

\begin{thebibliography}{10}

\bibitem{Aus15}
T.~Austin.
\newblock Additivity properties of sofic entropy and measures on model spaces.
\newblock preprint, http://arxiv.org/abs/1510.02392, 2015.

\bibitem{Bow10}
L.~Bowen.
\newblock Measure conjugacy invariants for actions of countable sofic groups.
\newblock {\em J. Amer. Math. Soc.}, pages 217--245, 2010.

\bibitem{Bow12a}
L.~Bowen.
\newblock Sofic entropy and amenable groups.
\newblock {\em Ergodic Theory and Dynamical Systems}, 32(2):427--466, 2012.

\bibitem{CovTho06}
Thomas~M. Cover and Joy~A. Thomas.
\newblock {\em Elements of information theory}.
\newblock Wiley-Interscience [John Wiley \& Sons], Hoboken, NJ, second edition,
  2006.

\bibitem{DGRS08}
A.~Dooley, V.~Golodets, D.~Rudolph, and S.~Sinel'shchikov.
\newblock Non-{B}ernoulli systems with completely positive entropy.
\newblock {\em Ergodic Theory and Dynamical Systems}, 28:87--124, 2007.

\bibitem{GolSin00}
Valentin~Ya. Golodets and Sergey~D. Sinel'shchikov.
\newblock Complete positivity of entropy and non-{B}ernoullicity for
  transformation groups.
\newblock {\em Colloq. Math.}, 84/85(part 2):421--429, 2000.
\newblock Dedicated to the memory of Anzelm Iwanik.

\bibitem{K}
A.S. Kechris.
\newblock {\em Global aspects of ergodic group actions}, volume 160 of {\em
  {M}athematical {S}urveys and {M}onographs}.
\newblock {A}merican {M}athematical {S}ociety, 2010.

\bibitem{Ke13}
D.~Kerr.
\newblock Sofic measure entropy via finite partitions.
\newblock {\em Groups Geom. Dyn.}, 7:617--632, 2013.

\bibitem{KeLi11}
D.~Kerr and H.~Li.
\newblock Entropy and the variational principle for actions of sofic groups.
\newblock {\em Inventiones Mathematicae}, 186:501--558, 2011.

\bibitem{Ke15}
David Kerr.
\newblock Bernoulli actions of sofic groups have completely positive entropy.
\newblock {\em Israel J. Math.}, 202(1):461--474, 2014.

\bibitem{RudWei00}
D.~Rudolph and B.~Weiss.
\newblock Entropy and mixing for amenable group actions.
\newblock {\em Annals of Mathematics}, 151(3):1119--1150, 2000.

\bibitem{Sew14a}
B.~Seward.
\newblock Krieger's finite generator theorem for ergodic actions of countable
  groups {I}.
\newblock preprint, http://arxiv.org/abs/1405.3604, 2014.

\bibitem{Sew14b}
B.~Seward.
\newblock Krieger's finite generator theorem for ergodic actions of countable
  groups {I}{I}.
\newblock preprint, http://arxiv.org/abs/1501.03367, 2014.

\bibitem{Wei03}
Benjamin Weiss.
\newblock Actions of amenable groups.
\newblock In {\em Topics in dynamics and ergodic theory}, volume 310 of {\em
  London Math. Soc. Lecture Note Ser.}, pages 226--262. Cambridge Univ. Press,
  Cambridge, 2003.

\end{thebibliography}

Courant Institute of Mathematical Sciences\\
New York University\\
New York NY, 10012\\
\texttt{tim@cims.nyu.edu}\\
\\
Department of Mathematics\\
California Institute of Technology\\
Pasadena CA, 91125\\
\texttt{pjburton@caltech.edu}

\end{document}